\documentclass[a4paper, reqno]{amsart}
\usepackage{amsmath, amssymb, amscd, bbm}
\usepackage{multicol, caption}
\usepackage{pifont}

\newtheorem{thm}{Theorem}[section]
\newtheorem{lemma}[thm]{Lemma}

\newtheorem{cor}[thm]{Corollary}
\newtheorem{defn}[thm]{Definition}

\theoremstyle{definition}
\newtheorem{example}[thm]{Example}

\newcommand{\R}{\ensuremath{\mathbbm{R}}}
\newcommand{\N}{\ensuremath{\mathbbm{N}}}
\newcommand{\Q}{\ensuremath{\mathbbm{Q}}}

\newcommand{\Z}{\ensuremath{\mathbbm{Z}}}

\newcommand{\ts}{\hspace{0.5pt}}

\newcommand{\ii}{\ts\mathrm{i}\ts}

\newcommand{\Lam}{\varLambda}
\newcommand{\Sig}{\varSigma}

\newcommand{\CS}{\mathcal{S}}

\DeclareMathOperator{\OC}{OC}
\DeclareMathOperator{\OS}{OS}

\DeclareMathOperator{\OG}{O}

\DeclareMathOperator{\Scal}{Scal}
\DeclareMathOperator{\scal}{scal}

\DeclareMathOperator{\lcm}{lcm}
\DeclareMathOperator{\den}{den}

\begin{document}

\title{Similar submodules and coincidence site modules}

\author{P. Zeiner}
\address{Fakult\"at f\"ur Mathematik, Universit\"at Bielefeld, 33615 Bielefeld, Germany}
\email{pzeiner@math.uni-bielefeld.de}

\begin{abstract} 
We consider connections between similar sublattices and coincidence site
lattices (CSLs),
and more generally between similar submodules and coincidence
site modules of general (free) $\Z$-modules in $\R^d$.

In particular, we generalise results obtained by S.~Glied and
M.~Baake~\cite{sglied1,sglied2} on
similarity and coincidence isometries of lattices and certain
lattice-like modules called $\CS$-modules.
An important result is that the factor group $\OS(M)/\OC(M)$ is Abelian
for arbitrary $\Z$-modules $M$,
where $\OS(M)$ and $\OC(M)$ are the groups of similar and coincidence
isometries, respectively. In addition, we derive various relations
between the indices of CSLs and their corresponding
similar sublattices.
\end{abstract}

\maketitle

\begin{multicols}{2}

\section{Introduction}

Coincidence site lattices (CSLs) are an important tool in describing grain
boundaries in crystals; see~\cite{frie11,KW, rang66, boll70, gribo74}
and references therein. These concepts have been generalised for modules to
analyse grain boundaries in quasicrystals~\cite{plea96,baa97,war93,warlue94}.
On the other hand, similar sublattices and submodules have been
studied~\cite{consloa99,baagri04,baaheu1}, and it soon turned that there
must be close connections between these two types of sublattices,
compare for instance~\cite{baaheu1} and~\cite{pzcsl4} for similar sublattices
and CSLs of the $A_4$-lattice. In 2008, S.~Glied and M.~Baake established
a connection between similar sublattices and CSLs by showing
that the group of coincidence isometries is a normal subgroup
of the group of similarity isometries~\cite{sglied1}, a result which was later
generalised to a certain class of modules~\cite{sglied2}, which the author
called $\CS$-modules. 

In this paper, we want to have a closer look at these connections.
In the first part, we elaborate in more detail on the connections
between similar sublattices and CSLs by proving some relation between
the coincidence index and the so-called denominator of a coincidence
isometry. In the second part, we present a generalisation of the results
by S.~Glied and M.~Baake to general $\Z$-modules.

Let us fix some notations and recall the most important notions first,
for more details we refer to~\cite{baa97,sglied1}.
Throughout this paper, $\Lam\subset \R^d$ denotes a lattice of full
rank in $\R^d$. An isometry $R\in\OG(d,\R)$, i.e. an orthogonal
transformation in $\R^d$, is called a 
\emph{coincidence isometry} of $\Lam$ if the intersection
$\Lam \cap R\Lam$ is a sublattice of $\Lam$ of full rank. This happens
if and only if the index $\Sig^{}_\Lam(R):=[\Lam: \Lam \cap R\Lam]$,
the so-called \emph{coincidence index}, is finite. In this case,
we call $\Lam(R):=\Lam \cap R\Lam$ a \emph{coincidence site lattice} (CSL).
The set of all coincidence isometries forms a group, which we denote
by $\OC(\Lam)$.

Two lattices $\Lam_1$ and $\Lam_2$ are called \emph{commensurate},
denoted by $\Lam_1 \sim \Lam_2$,
if $\Lam_1 \cap \Lam_2$ is a sublattice of full rank
of both $\Lam_1$ and $\Lam_2$.
As we assume throughout this paper that any lattice has full rank,
$\Lam_1$ and $\Lam_2$ are commensurate if and only if $\Lam_1 \cap \Lam_2$
is a sublattice of at least one of $\Lam_1$ and $\Lam_2$. Equivalently,
$\Lam_1$ and $\Lam_2$ are commensurate if and only if there exists an integer
$m$ such that $m \Lam_1 \subset \Lam_2$. Thus, $R$ is a coincidence isometry
of $\Lam$ if and only if $\Lam$ and $R\Lam$ are commensurate.

A \emph{similar sublattice} (SSL) is a sublattice of $\Lam$ that is similar
to~$\Lam$, i.e. it is a sublattice of the form $\alpha R \Lam \subset \Lam$
for some $R\in\OG(d,\R)$ and $\alpha\in\R^+$. We call $R$ a similarity
isometry of $\Lam$, if there exists an $\alpha\in\R^+$ such that
$\alpha R \Lam$ is a similar sublattice. The set of similarity isometries
forms a group as well, which we denote by $\OS(\Lam)$. For any
$R\in\OS(\Lam)$ we define the \emph{denominator} $\den^{}_\Lam(R)$
as the smallest
scaling factor $\alpha\in\R^+$ such that $\alpha R \Lam \subset \Lam$.
Recall that $\alpha^d$, and thus $\den^{}_\Lam(R)^d$, is an integer.

For any $R\in\OG(d,\R)$ we can define the two sets
\begin{align}
  \Scal^{}_\Lam(R)&:=\{\alpha\in\R \mid \alpha R \Lam\subseteq \Lam \}, \\
  \scal^{}_\Lam(R)&:=\{\alpha\in\R \mid \alpha R \Lam\sim \Lam \}, 
\end{align}
where the first one consists of all scaling factors giving rise
to a similar sublattice and the latter is the set of all scaling factors
which lead to lattices commensurate to~$\Lam$.
Observe that $\Scal^{}_\Lam(R)$ and $\scal^{}_\Lam(R)$ are non-trivial
if and only if $R\in \OS(\Lam)$. In other words,
$\Scal^{}_\Lam(R)\ne \{0\}$, and likewise $\scal^{}_\Lam(R)\ne \{0\}$,
if and only if $R\in \OS(\Lam)$, compare~\cite{sglied1,sglied2}.
In particular, if $E$ is the identity operation, then
$\Scal^{}_\Lam(E)=\Z$ and $\scal^{}_\Lam(E)=\Q$. More generally,
\begin{align}
\Scal^{}_\Lam(R)=\den^{}_\Lam(R)\,\Z  \label{eq:Scal-lat}\\
\scal^{}_\Lam(R)=\den^{}_\Lam(R)\,\Q. \label{eq:scal-lat}
\end{align}

\section{Similar sublattices and CSLs}

As mentioned above, there is a close connection between the groups
$\OC(\Lam)$ and $\OS(\Lam)$. In particular, S.~Glied and M.~Baake
have shown the following~\cite{sglied1} result.
\begin{thm}\label{thm:OCOS-lat}
  The kernel of the homomorphism 
  \begin{align}
    \phi: \OS(\Lam) & \to \R^+/\Q^+,\nonumber \\
    R & \mapsto \scal^{}_\Lam(R) \cap \R^+ \nonumber
  \end{align}
  is the group $\OC(\Lam)$. Thus $\OC(\Lam)$ is a normal subgroup of
  $\OS(\Lam)$ and $\OS(\Lam)/\OC(\Lam)$ is Abelian. 
  
  Moreover, all elements of $\OS(\Lam)/\OC(\Lam)$ have finite order, 
  in particular, their order is a divisor of the dimension~$d$.
\end{thm}
Hence, any coincidence isometry is a similarity isometry, and thus
it makes sense to compare $\Sig(R)$ and $\den(R)$ for any $R\in\OC(\Lam)$.
By definition, $\Sig(R)$ is a positive integer, and so is
$\den(R)$ for any $R\in\OC(\Lam)$. This can be seen as follows: by
Theorem~\ref{thm:OCOS-lat}, we see $\scal^{}_\Lam(R)=\Q$, which reflects
the fact that $\Lam$ and $R\Lam$ are commensurate. Thus
$\den(R)\in\Scal^{}_\Lam(R)\subset\scal^{}_\Lam(R)=\Q$, and as $\den(R)^d\in\Z$,
we see $\den(R)\in\N$. 

Recall that $\Sig(R^{-1})=\Sig(R)$ for any $R\in\OC$, compare~\cite{baa97},
whereas $\den(R^{-1})$ and $\den(R)$ are not equal in
general~\cite{sglied1,pzinprep}. Nevertheless $\den(R^{-1})$ and $\den(R)$
are not independent of each other, as we will show in a moment. But first we
mention
\begin{lemma}\label{lem:den-mult-lat}
  Let $\Lam$ be a lattice in $R^d$. Then, for any $R,S\in\OS(\Lam)$,
  \begin{align}  \label{eq:den-mult-lat}
    \frac{\den(R)\den(S)}{\den(RS)}\in\N.
  \end{align}
\end{lemma}
\begin{proof}
  By the definition of the denominator,
  \begin{align*}
    \den(R)\den(S) RS \Lam
    & =  \den(R) R \left( \den(S) S \Lam \right)\\
    & \subseteq \den(R) R \Lam \subseteq \Lam
  \end{align*}
  which is only possible if $\den(R)\den(S)\in\Scal(RS)=\den(RS)\Z$.
  As the denominator
  is positive by definition, Eq.~\eqref{eq:den-mult-lat} follows.
\end{proof}
Now we are able to prove the following relations between
$\den(R^{-1})$ and $\den(R)$.
\begin{lemma}\label{lem:den-lat}
  Let $\Lam$ be a lattice in $\R^d$. Then, for any $R\in\OS(\Lam)$,
  \begin{align}
    \den(R)\den(R^{-1}) & \in \N \label{eq:den-lat1}\\
    \frac{\den(R)^{d-1}}{\den(R^{-1})} & \in \N \label{eq:den-lat2}
  \end{align}
\end{lemma}
\begin{proof}
  Eq.~\eqref{eq:den-lat1} is just a special case of
  Lemma~\ref{lem:den-mult-lat}, with $S=R^{-1}$ and $\den(E)=1$.

  For the second claim, observe that $\den(R)R\Lambda$ has index
  $[\Lam:\den(R)R\Lambda]=\den(R)^d$ in $\Lam$. Thus,
  $\den(R)^d\Lam \subseteq \den(R)R\Lambda$, or equivalently,
  $\den(R)^{d-1} R^{-1} \Lam \subseteq \Lambda$. This means
  $\den(R)^{d-1}\in \Scal^{}_\Lam(R)$, and now Eq.~\eqref{eq:Scal-lat} implies
  Eq.~\eqref{eq:den-lat1}.
\end{proof}
For $d=2$ this result simplifies considerably.
\begin{cor}
 $\den^{}_\Lam(R^{-1})=\den^{}_\Lam(R)$ for any planar lattice $\Lam$ and any
 $R\in\OS(\Lam)$.
\end{cor}

We are now able to prove the following bounds for $\Sig(R)$ in terms
of the denominators $\den(R)$ and $\den(R^{-1})$.
\begin{thm}\label{theo:den-Sig1}
  Let $\Lam$ be a lattice in $\R^d$. Then, for any $R\in\OC(\Lam)$,
  \begin{enumerate}
  \item \label{enu:den-Sig1}
    $\lcm\left(\den^{}_\Lam(R),\den^{}_\Lam(R^{-1})\right)$ divides 
    $\Sig^{}_\Lam(R)$,
   \item \label{enu:den-Sig2}
    $\Sig^{}_\Lam(R)$ divides 
    $\gcd\left(\den^{}_\Lam(R),\den^{}_\Lam(R^{-1})\right)^d$,
  \item \label{enu:den-Sig3}
    $\Sig^{}_\Lam(R)^2$ divides 
    $\lcm\left(\den^{}_\Lam(R),\den^{}_\Lam(R^{-1})\right)^d$.
  \end{enumerate}
\end{thm}
\begin{proof}
  For~\eqref{enu:den-Sig1} recall
  that $\Lam(R)$ has index $\Sig(R)$ in $\Lam$, thus
  $\Sig(R) \Lam \subseteq \Lam(R) \subseteq R \Lam$, or equivalently,
  $\Sig(R) R^{-1} \Lam \subseteq \Lam$. An argument as above shows
  that $\Sig(R)$ is a multiple of $\den(R^{-1})$. By symmetry, $\den(R)$
  is a divisor of $\Sig(R^{-1})=\Sig(R)$ as well, and
  hence~\eqref{enu:den-Sig1} follows.

  For~\eqref{enu:den-Sig2} we exploit that $\den(R)$ is an integer for
  $R\in\OC(\Lam)$. Thus $\den(R) R \Lam$ is a sublattice
  of both $\Lam$ and $R \Lam$, and hence $\den(R) R \Lam \subseteq \Lam(R)$.
  Comparing the indices of $\den(R) R \Lam$ and $\Lam(R)$
  in $\Lam$ shows that $\Sig(R)$ divides $\den(R)^d$. Using
  $\Sig(R^{-1})=\Sig(R)$ as above finally yields~\eqref{enu:den-Sig2}.
  
  Finally, let $a:=\lcm\left(\den(R),\den(R^{-1})\right)$. Then $a \Lam$ and
  $a R \Lam$ are both sublattices of $\Lam$ and $R \Lam$, hence
  $a (\Lam + R \Lam)$ is a sublattice of $\Lam \cap R \Lam$ with index
  \[
    [R \cap R \Lam: a (\Lam + R \Lam)]= \frac{a^d}{\Sig(R)^2},
  \]
  as
  $\Sig(R)=[\Lam: \Lam(R)]=[\Lam + R \Lam:\Lam]$. Hence $\Sig(R)^2$
  divides $a$.
\end{proof}
The situation becomes particularly simple for planar lattices, where
we get the following result by recalling $\den^{}_\Lam(R)=\den^{}_\Lam(R^{-1})$.
\begin{cor}\label{theo:den-Sig2D}
  Let $\Lam$ be a lattice in $\R^2$. Then, for any $R\in\OC(\Lam)$,
  \begin{align}\label{eq:sig-den-lat2dim}
    \Sig^{}_\Lam(R)=\den^{}_\Lam(R).
  \end{align}
\end{cor}
This results turns out to be very useful in the analysis of CSLs
of planar lattices, as the denominator is usually much simpler to determine
than the coincidence index.

In more than two dimensions Eq.~\eqref{eq:sig-den-lat2dim} is in general not
true anymore~\cite{pzcsl4,baa97,pzcsl2}, although it may be satisfied
in special cases, e.g. Eq.~\eqref{eq:sig-den-lat2dim} holds for all
coincidence isometries of the cubic lattices~\cite{gribo74,baa97,pzcsl1}.

One word of caution should be added. Although there are a lot of connections
between CSLs and similar sublattices, a CSL is in general not a similar
sublattice, see~\cite{pzcsl1} for the cubic case. An exception are the square
and hexagonal lattices~\cite{plea96,baa97}, where every CSL is a similar
sublattice. In these two cases any coincidence rotation $R$ is the square
of a suitable $S\in\OS(\Lam)$ and we have $\Lam(R)=\den(S) S \Lam$. 
But note that we cannot have $\Lam(R)=\alpha R \Lam$ except for symmetry
operations $R$, since index considerations immediately imply
$\alpha=\den(R)=\den(R^{-1})$, which would give
$\Sigma(R)=\lcm\left(\den(R),\den(R^{-1})\right)^d$, which contradicts
part~\eqref{enu:den-Sig3} of Theorem~\ref{theo:den-Sig1}.

\section{The module case}

For the description of quasicrystals we need to go beyond lattices. The
right tool here are special kinds of $\Z$-modules, namely free $\Z$-modules
in $\R^d$, which we assume to span $\R^d$. More precisely, let
$t_1,\ldots,t_k\in\R^d$ be rationally independent vectors that
span $\R^d$, i.e., $\langle t_1,\ldots,t_k\rangle_\R=\R^d$. Then
\[
  M:=\langle t_1,\ldots,t_k\rangle_\Z
  =\{n_1 t_1+\ldots+n_kt_k \mid n_k\in\Z\}\subseteq \R^d
\]
is called a (free) \emph{$\Z$-module} of rank $k$ in dimension $d$.

Throughout this paper we shall call these $\Z$-modules simply modules.
Clearly, $M$ is a lattice if and only if $k=d$. Only in this case $M$
forms a discrete subset of $\R^d$.

$M$ is a free Abelian group of rank $k$, i.e., it is isomorphic to $\Z^k$.
In fact, $M$ can always be obtained as a projection of a $k$-dimensional
lattice into $\R^d$. From an algebraic point of view, lattices and free
$\Z$-modules of finite rank are the same, and we thus expect that we can
generalise the concepts and results for similar sublattices and CSLs
easily to the case of modules. However, some care is needed as this problem
is not a purely algebraic problem but also involves geometry. In fact,
a key ingredient are orthogonal and similarity transformations in $\R^d$,
which induce linear transformations in $\R^k$, but the latter need not
be orthogonal or similarity transformations, respectively. In addition,
the fact that $M$ is not discrete (except for lattices) may cause
some problems. So we have to carefully check
which concepts and results we can generalise.

We start with some definitions that are straightforward generalisations
of the lattice case. We again restrict to submodules that have full rank,
i.e. we call a module $M_1\subseteq M$ a \emph{submodule} of
\mbox{$M\subseteq \R^d$},
if and only if it has full rank, or equivalently,
if the index $[M:M_1]$ is finite.
  
\begin{defn}
  Two modules $M_1, M_2 \subseteq \R^d$ are called \emph{commensurate}, 
  denoted by $M_1\sim M_2$, if
  $M_1\cap M_2$ is a submodule of both $M_1$ and $M_2$.
\end{defn}
In other words, two modules $M_1, M_2 \subseteq \R^d$ are called
commensurate, if
there exists an $m\in\N$ such that $mM_1\subseteq M_2$ and $mM_2\subseteq M_1$.

Two modules $M_1, M_2 \subseteq \R^d$ are called \emph{similar}, if there
exists a \emph{similarity transformation} between them, i.e., there exist
$\alpha\in \R$ and $R\in \OG(d,\R)$ such that $M_1 = \alpha R M_2$.
\begin{defn}
  $M_1$ is called a \emph{similar submodule (SSM)} of $M$, if  there exist
  $\alpha\in \R$ and $R\in \OG(d, \R)$ such that $M_1 = \alpha R M$.
\end{defn}

As in the lattice case we define the set of \emph{similarity isometries} by
\[
  \OS(M):=\{R\in \OG(d,\R)\mid \exists \alpha\in\R^+ \mbox{ with } 
    \alpha R M \subseteq M \}.
\]
As is to be expected from the lattice case we have
\begin{thm}
  $\OS(M)\subseteq \OG(d,\R)$ is a group.
\end{thm}

The next quantity to look at are the scaling factors. In particular,
it is the sets of scaling factors that are crucial for the understanding
of SSMs and their relation to CSMs. Hence we define
\begin{align}
   \Scal^{}_M(R)&:=\{\alpha\in\R \mid \alpha R M\subseteq M \}, \nonumber\\
   \scal^{}_M(R)&:=\{\alpha\in\R \mid \alpha R M\sim M \}. \nonumber
\end{align}
We have already encountered them in the
lattice case, but there their importance may not have been so clear as they
were just multiples of the sets $\Z$ and $\Q$.
Again, $\Scal^{}_M(R)$ and $\scal^{}_M(R)$ are non-trivial if and only if
$R\in \OS(M)$, i.e. $\Scal^{}_M(R)\ne \{0\}$, and $\scal^{}_M(R)\ne \{0\}$,
if and only if $R\in \OS(M)$.

Naturally, there are some restrictions on the possible values of $\alpha$.
We have seen that $\alpha^d\in\Z$ for any $\alpha \in \Scal^{}_\Lam(R)$
in the case of lattices. More generally, one can show
$\alpha^d\in\CS$, if $M$ is an $\CS$-module~\cite{sglied2}.
  
In general, the situation is more complex, and the crucial quantity
is the rank $k$.
\begin{thm}\label{thm:alphaScalM}
  Any $\alpha\in \Scal^{}_M(R)$ is an algebraic integer. 
  If $M$ has rank $k$, then $\alpha$ has degree at most $k(k-1)$.
\end{thm}
For lattices and $\CS$-modules the degree of $\alpha$ is bounded by~$k$,
whereas in general the upper bound $k(k-1)$ cannot be improved,
as is shown by the following example.
\begin{example}\label{ex:eta}
  Let $\eta= e^{\frac{\ii \pi}{3}} \sqrt[3]{\tau} 
  - e^{-\frac{\ii \pi}{3}} \frac{1}{\sqrt[3]{\tau}}$, where
  $\tau=\frac{1+\sqrt{5}}{2}$ is the golden mean. Then
  $M=\Z[\eta]$ has rank $3$, as $\eta$ satisfies
  $\eta^3 + 3 \eta - 1 = 0$. 
  Here, $\eta=|\eta|\frac{\eta}{|\eta|}$ is a symmetry
  operation, whose scaling factor has degree $6=3\cdot 2$.
\end{example}

Let us go a step further and ask which properties the sets
$\Scal(R)$ and $\scal(R)$ have.
To begin with, we consider $\Scal(E)$, which gives the ``trivial'' CSMs.
As $\Scal(E)=\Z$ for lattices, and  \mbox{$\Scal(E)=\CS$} for $\CS$-modules,
we expect $\Scal(E)$ to be a ring of algebraic integers.
\begin{thm}
  Let $M \subseteq \R^d$ be a free $\Z$-module of rank $k$.
  $\Scal^{}_M(E)$ is a ring of some algebraic integers. In particular, \/
  $\Scal^{}_M(E)$
  is a ring with unity and it is a finitely generated free $\Z$-module,
  whose rank is a divisor of $k$ and is at most $\frac{k}{d}$.
  Moreover, $\scal(E)$ is the corresponding field of quotients.
\end{thm}

Hence the modules $M$ are not only free $\Z$-modules, but also
$\Scal(E)$-modules, but in general \emph{not} free ones, as is shown
by the following example.
\begin{example}
  Let $\xi_8=e^{\ii \pi/4}$. Then $M=\langle 1,\ii,2\xi_8,-2\bar\xi_8 \rangle_\Z$
  is a submodule of $\Z[\xi_8]$ of index $4$. In particular, \/
  $\Scal^{}_M(E)=\Z[2\sqrt{2}]$, which is not a principle ideal domain (PID).
  $M$ is not a free
  $\Scal^{}_M(E)$-module and thus not an $\CS$-module in the sense
  of~\cite{sglied2}.
\end{example}

We now look at $\Scal(R)$ and $\scal(R)$
for general $R$. We start with $\scal(R)$, as the results are much nicer for
$\scal(R)$, which is due to
the fact that $\scal(E)$ is a field.
\begin{thm}\label{theo:scalR}
  Let $R,S\in \OS(M)$ and let $\alpha$ be an arbitrary element
  of $\scal^{}_M(R)$. Then, \pagebreak[3]
  \begin{align}
    \scal^{}_M(R)&=\alpha \scal^{}_M(E),  \\
    \scal^{}_M(RS)&=\scal^{}_M(R)\scal^{}_M(S). 
  \end{align}
\end{thm}
In fact, this is in some sense a generalisation of
Lemma~\ref{lem:den-mult-lat} and Eq.~\eqref{eq:scal-lat}. In particular,
it follows that
\begin{align}
  \scal^{}_M(R)\scal^{}_M(R^{-1})=\scal^{}_M(E).
\end{align}
Thus, the set $\{ \scal(R) \mid R\in\OS(M) \}$ has a natural group structure,
with unit element $\scal^{}_M(E)$. It is isomorphic to a (countable) subgroup
of a factor group of the multiplicative group $(\R^+,\cdot)$.
It will turn out later that this
group plays a fundamental role in connecting $\OC(M)$ and $\OS(M)$.

For $\Scal(R)$ the situation is more complex, and the generalisation
of Eq.~\eqref{eq:Scal-lat} reads as follows.
\begin{thm}
  $\Scal^{}_M(R)$ is a finitely generated free $\Z$-module. 
  Moreover, $\beta\Scal^{}_M(R)\subseteq \Scal^{}_M(R)$ for
  any $\beta\in \Scal^{}_M(E)$ i.e., $\Scal^{}_M(R)$ is also a finitely
  generated \/ $\Scal^{}_M(E)$-module.

  If \/ $\Scal^{}_M(E)$ is a PID, \/ $\Scal^{}_M(R)$ is a free \/
  $\Scal^{}_M(E)$-module of
  rank $1$, i.e.\ there exists an $\alpha\in \Scal^{}_M(R)$ such that \/
  $\Scal^{}_M(R)= \alpha \Scal^{}_M(E)$.
\end{thm}
Thus defining a denominator makes only sense if $\Scal^{}_M(E)$ is a PID.
Nevertheless one can generalise Lemma~\ref{lem:den-mult-lat}
and~\ref{lem:den-lat} to a certain extent using ideals~\cite{pzinprep}.
However,
we do not want to pursue this topic here any further, as some complications
arise if $M$ is not a free $\Scal^{}_M(E)$-module.

Let us turn our attention to coincidence site modules now,
which are defined as follows.
\begin{defn}
  Let $R\in \OG(d,\R)$. If $M$ and $R M$ are commensurate, 
  $ M(R):=M \cap R M$ is called a \emph{coincidence site module (CSM)}.
  In this case, $R$ is called a \emph{coincidence isometry}.
  The corresponding index $\Sig^{}_M(R):=\mbox{$[M: M(R)]$}$ is called its
  \emph{coincidence index}.
\end{defn}

In complete analogy to the lattice case we have the following result.
\begin{thm}
  The set of all coincidence isometries
  \[
    \OC(M):=\{R\in \OG(d,\R)\mid M\sim R M \}
  \]
  forms a group, a subgroup of \/ $\OG(d,\R)$.
\end{thm}

As $M$ is not discrete in general, we cannot define a unit cell with a
non-zero volume. Thus, the proof of the following result becomes more
complicated, as one needs algebraic methods instead of the usual argument
of the preservation of volume~\cite{pzinprep}.
\begin{thm}\label{theo:Sigma-Rinv}
  For any $R\in \OC(M)$
  \[
    \Sig^{}_M(R)=\Sig^{}_M(R^{-1}).
  \]
\end{thm}

As for lattices, the group $\OC(M)$ can be characterised by $\scal_M(R)$.
\begin{lemma}\label{lem:OC-scal}
  Let $M\subseteq \R^d$ be a finitely generated $\Z$-module and let $\OG(M)$
  be its symmetry group. Then
  \begin{enumerate}
    \item $R\in \OC(M)$ if and only if $1 \in \scal_M(R)$.
    \item $R\in \OG(M)$ if and only if $1 \in \Scal_M(R)$.
  \end{enumerate}
\end{lemma}

By Theorem~\ref{theo:scalR}, $\{ \scal(R) : R\in \OS(M)\}$ forms a
group with unit element $\scal(E)$. This gives us the following analogue of
Theorem~\ref{thm:OCOS-lat}
\begin{thm}
  The kernel of the homomorphism 
  \begin{align}
    \phi: \OS(M) & \to \R^+/(\scal^{}_M(E)\cap\R^+),\nonumber \\
    R & \mapsto \scal^{}_M(R) \cap \R^+ \nonumber
  \end{align}
  is the group $\OC(M)$. Thus $\OC(M)$ is a normal subgroup of $\OS(M)$
  and $\OS(M)/\OC(M)$ is Abelian.
\end{thm}
We have seen that in the lattice case all elements of
$\OS(M)/\OC(M)$ have finite order, which is a divisor
of~$d$. This is no longer true for
general modules $M$, see Example~\ref{ex:eta}, where there are no non-trivial
coincidence isometries and any element of $\OS(M)/\OC(M)$, except the unit
element, has infinite order.


\section*{Acknowledgement}

The author thanks M.~Baake and S.~Glied for interesting discussions.
This work was supported by the German Research Council (DFG),
within the CRC 701.

\end{multicols}

\begin{thebibliography}{99}

\bibitem{frie11}
G.~Friedel, \emph{Le\c{c}ons de Cristallographie},
Blanchard, Paris 1911.
		
\bibitem{KW} 
M.L. Kronberg, F.H. Wilson, Secondary recrystallization in copper,
\emph{Trans. A.I.M.E.} \textbf{185}, 501--514 (1949).

\bibitem{rang66}
S. Ranganathan, On the Geometry of Coincidence-Site Lattices,
\emph{Acta Cryst.} \textbf{21}, 197--199 (1966).
DOI:10.1107/S0365110X66002615

\bibitem{boll70} 
W. Bollmann, \emph{Crystal Defects and Crystalline Interfaces},
Springer, Berlin 1970.

\bibitem{gribo74}
H. Grimmer, W. Bollmann, D. H. Warrington, 
Coincidence-Site Lattices and Complete Pattern-Shift Lattices in Cubic
Crystals,
\emph{Acta Cryst.} \textbf{A30}, 197--207 (1974).
DOI: 10.1107/S056773947400043X

\bibitem{plea96}
P.A.B. Pleasants, M. Baake, J. Roth, 
Planar coincidences for ${N}$--fold symmetry,
\emph{J. Math. Phys.} \textbf{37}, 1029--1058 (1996).
DOI: 10.1063/1.531424

\bibitem{baa97} 
M. Baake, Solution of the
Coincidence Problem in Dimensions $d\leq 4$,
in \emph{The Mathematics of Long-Range Aperiodic Order},
Eds. R.V.~Moody, Kluwer, Dordrecht 1997, p.~9--44.

\bibitem{war93} 
D.H. Warrington, Coincidence site lattices in quasicrystal tilings
\emph{Mat. Science Forum} \textbf{126-128}, 57--60 (1993).
DOI: 10.4028/www.scientific.net/MSF.126-128.57

\bibitem{warlue94} 
D.H. Warrington, R.~L\"{u}ck, 
in \emph{Proc. Intl. Conf. on Aperiodic Crystals (Les Diablerets)},
Eds. G. Chapuis, W. Paciorek, World Scientific, Singapore 1994, p.~30--34.

\bibitem{consloa99}
J.H.~Conway, E.M.~Rains, N.J.A.~Sloane,
On the existence of similar sublattices,
\emph{Can.\ J.\ Math.} \textbf{51}, 1300--1306 (1999).
DOI: 10.4153/CJM-1999-059-5

\bibitem{baagri04}
M. Baake, U. Grimm,
Bravais colourings of planar modules with ${N}$--fold symmetry,
\emph{Z. Kristallogr.} \textbf{219}, 72--80 (2004).
DOI: 10.1524/zkri.219.2.72.26322

\bibitem{baaheu1}
M. Baake, M. Heuer, R.V. Moody,
Similar sublattices of the root lattice ${A}_4$,
\emph{J.\ Algebra} \textbf{320}, 1391--1408 (2008).
DOI: 10.1016/j.jalgebra.2008.04.021

\bibitem{pzcsl4}
M. Baake, M. Heuer, U. Grimm, P. Zeiner,
Coincidence rotations of the root lattice ${A}_4$,
\emph{Europ. J. Combinatorics} \textbf{29}, 1808-1819 (2008).
DOI: 10.1016/j.ejc.2008.01.012

\bibitem{sglied1}
S.~Glied, M.~Baake, 
Similarity versus coincidence rotations of lattices,
\emph{Z.\ Krist.} \textbf{223}, 770--772 (2008).
DOI: 10.1524/zkri.2008.1054

\bibitem{sglied2}
S.~Glied, 
Similarity and coincidence isometries for modules,
\emph{Can.\  Math.\ Bull.} \textbf{55}, 98--107 (2011).
DOI: 10.4153/CMB-2011-076-x

\bibitem{pzinprep}
P.~Zeiner,
in preparation

\bibitem{pzcsl2}
P. Zeiner, Coincidences of hypercubic lattices in 4 dimensions,
\emph{Z. Kristallogr.} \textbf{221}, 105--114 (2006).
DOI: 10.1524/zkri.2006.221.2.105

\bibitem{pzcsl1}
P. Zeiner, Symmetries of coincidence site
lattices of cubic lattices,
\emph{Z. Kristallogr.} \textbf{220}, 915--925 (2005).
DOI: 10.1524/zkri.2005.220.11\_2005.915

\end{thebibliography}
\end{document}